\documentclass{lms}
\usepackage{amsmath, amsfonts}
\usepackage{graphicx}
\usepackage{url}

\def\R{\mathbb{R}}

\def\Sp{\mathbb{S}}

\def\i{\indent}

\def\L{\Lambda}

\def\G{\Gamma}

\def\om{\omega}
\def\s{\sigma}

\def\m{\mu}
\def\n{\nu}
\def\l{\lambda}

\def\half{\frac{1}{2}}

\def\half{\frac{1}{2}}

\newcommand\EatDot[1]{}

\newtheorem{theorem}{Theorem}[section] % 1st argument is your name for it
\newtheorem{corollary}[theorem]{Corollary}

\newnumbered{assertion}{Assertion}    % 1st argument is your name for it
\newnumbered{conjecture}{Conjecture}  % 2nd argument is what is printed
\newnumbered{definition}{Definition}
\newnumbered{hypothesis}{Hypothesis}
\newnumbered{remark}{Remark}
\newnumbered{note}{Note}
\newnumbered{observation}{Observation}
\newnumbered{problem}{Problem}
\newnumbered{question}{Question}
\newnumbered{algorithm}{Algorithm}
\newnumbered{example}{Example}
\newunnumbered{notation}{Notation} % This is usually unnumbered
\title[Representations of Ultraspherical Polynomials]% end with percent
 {Integral Representations of Ultraspherical Polynomials II} % This is the full title of the paper
\author{N. H. Bingham and Tasmin L. Symons}

\dedication{In memory of Dick Askey, 4th June 1933 -- 9th October 2019}

\classno{42C05 (primary), 33C45 (secondary)}

\extraline{This research was conducted during the second author's PhD, funded by the EPSRC Centre for Doctoral Training in the Mathematics of Planet Earth [EP/L016613/1].
}

\begin{document}
\maketitle

\begin{abstract}
In the first part, by the first author's work of 1972, an integral representation for an ultraspherical polynomial of higher index in terms of one of lower index and an infinite series was obtained.  While this representation works well from a theoretical point of view, it is not numerically satisfactory as it involves polynomials of high degree, which are numerically unstable.  Here we sum this series to obtain an integral, which is numerically tractable.
\end{abstract}

%\part{Use this type of header for very long papers only}
% use lowercase except for proper names

\section{Introduction} % use lowercase except for proper names
\label{intro}

As in \cite{bin_rep}, we write $W_n^{\l}(x)$ for the ultraspherical (or Gegenbauer) polynomial of degree $n$ and index $\l > 0$ (see e.g. \cite[p. 302]{andrewsaskroy}, \cite[\S 4.4]{sze}), normalised so that $W_n^{\l}(1) = 1$.  This choice of normalisation, due to Bochner \cite{Boc_zonal}, is convenient probabilistically; the first part \cite{bin_rep} was probabilistically motivated \cite{bin_rws}; so too is this sequel (see \S 5).  The $W_n^{\nu}$ are orthogonal polynomials on the interval $[-1,1]$ with respect to the probability measure $G_{\nu}$, where
\begin{equation}
G_\nu(dy) = \frac{\Gamma(\nu + 1)}{\sqrt{\pi}\Gamma(\nu + 1/2)} (1-y^2)^{\nu-1/2} dy:  
\end{equation}
\[
\int_{-1}^1 W_m^{\nu} (x) W_n^{\nu}(x) G_{\nu}(dx) = {\delta}_{mn} {\om}^{\nu}_n,
\]
where
\begin{equation}
\om^\nu_m = \frac{n+\nu}{\nu} \frac{\Gamma(n + 2 \nu)}{n! \Gamma(2\nu)}.
\end{equation}   

\begin{theorem}[\cite{bin_rep}]\label{thrm:askfitch}
If $0 < \nu < \l$, $x \in [-1,1]$, there exists a probability measure $M_\nu^\l(x; dy)$ on $[-1,1]$ such that
\begin{equation}
W^\l_n (x) = \int_{-1}^1 W^\nu_n(y) \, M^\l_\nu(x; dy).
\end{equation}
Moreover, when $\l \neq \nu$ the measure $M^\l_\nu$ is absolutely continuous with density
\begin{equation}\label{bingham_density}
M^\l_\nu (x; dy) = G_\nu(dy) \sum_{m=0}^\infty \omega^\nu_m  W^\l_m(x) W^\nu_m(y).
\end{equation}
\end{theorem}

The series in \ref{bingham_density} is transparent from a theoretical point of view (it is derived in \cite{bin_rep} from the earlier work of Askey and Fitch \cite{askf} by an Abel-limit operation), but unsuitable for numerical use as it involves polynomials of high degree, which oscillate wildly. Our purpose here is to circumvent this by giving an explicit formula for the sum of the infinite series as a {\it double integral}, which is numerically tractable.  Our result, Theorem \ref{thrm:binint} below, is interesting in its own right, completing the integral representations in \cite{bin_rep} by showing the dependence on the higher index, $\lambda$, in a more convenient and structurally revealing way. 

\section{Preliminaries}
\label{prelim}

The Poisson kernel for the Jacobi polynomials reduces in the ultraspherical case to the generating function 
\begin{equation}\label{eqn:GF}
\sum_{n=0}^{\infty} {\om}_n^{\n} r^n W_n^{\n}(x) = (1 - r^2)/(1 - 2rx + x^2)^{\n + 1},
\quad r \in (-1,1),                                                                 
\end{equation}
cf. \cite[(2.1)]{bin_rep}.  Note that this is not the usual generating function for the ultraspherical polynomials \cite[\S 4.7.23]{sze}. 

Askey and Fitch \cite{askf} showed that for $x,y \in [-1,1]$, $r \in (-1,1)$, $0 \leq \nu < \lambda \leq \infty$, the series
\begin{equation}\label{eqn:AF-r}
\sum_{n=0}^{\infty} {\om}_n^{\nu} r^n W_n^{\lambda}(x)W_n^{\nu}(y)          
\end{equation}
converges to a non-negative sum-function, which leads to a corresponding probability measure $M^{\l}_{\nu}(x)$ satisfying
\begin{equation}\label{eq:ast}
W_n^{\l}(x) 
= \int_{-1}^1 W_n^{\nu}(y) M^{\l}_{\nu}(x; dy), \quad n = 0,1,2, \ldots. 
\end{equation}
Here (see \cite{bin_rep}) we may take $0 \leq \nu \leq \l \leq \infty$, $x \in [-1,1]$.  Some cases give Dirac laws: if $x = \pm 1$, $M^{\l}_{\nu}(\pm 1) = \delta_{\pm 1}$ (as $W_n^{\l}(\pm 1) = (\pm 1)^n$).  If $\l = \nu$, then $M^{\l}_{\l}(x) = {\delta}_x$ (as there is no projection to be done); so we may restrict to $\nu < \l$ as before. Now \cite[Lemma 1]{bin_rep} gives the Abel-limit operation explicitly: for $x, y \in (-1,1)$, we may take $r = 1$ here to get
\begin{equation}\label{eqn:AF-1}
m^{\l}_{\nu}(x; y)
:= \sum_{n=0}^{\infty} {\om}_n^{\nu} W_n^{\lambda}(x)W_n^{\nu}(y) \geq 0,  
\end{equation}
a non-negative function in $L_1(G_{\nu})$, finite-valued unless $x = y$ and $\nu < \lambda \leq \nu + 1$.  It is in fact the {\it Radon-Nikodym derivative} $dM^{\l}_{\nu}(x; dy)/dG_{\nu}(dy)$:
\begin{equation}\label{eq:RN}
M^{\l}_{\nu}(x; dy) = G_{\nu}(dy)\cdot m^{\l}_{\nu}(x; y) 
= G_{\nu}(dx)\cdot \sum_{n=0}^{\infty} {\om}_n^{\nu} W_n^{\lambda}(x)W_n^{\nu}(y).
\end{equation}

Following \cite{bin_rep}, for $\l > \n$ write $H_{\n}^{\l}$ for the probability measure of Beta type on $[0,1]$ given by the {\it Sonine law}
\begin{equation}
H_{\n}^{\l}(dx) := \frac{2 \G(\l + \half)}{\G(\n + \half) \G(\l - \n)} \cdot x^{2\n} (1 - x^2)^{\l - \n - \half} dx.
\end{equation}
This occurs in Sonine's first finite integral for the Bessel function \cite[p. 373]{wat1}: for
\begin{align}
\L_{\m}(t) &:= \G(\n + 1) J_{\n}(t)(t/2)^{-\m}, \\
\L_{\l - \half}(t) &= \int_0^1 \L_{\n - \half}(ut)H_{\n}^{\l}(du)    \label{eqn:S}
\end{align}
(the drop by a half-integer in parameter here reflects the drop in dimension in ${\Sp}^d \subset {\R}^{d+1}$; see \S 4 below). 

For the product of $W_n$ terms in (\ref{eqn:AF-1}), we need Gegenbauer's multiplication theorem for the ultraspherical polynomials \cite[p. 369]{wat1},
\begin{equation}\label{eqn:G}
W_n^{\n}(x) W_n^{\n}(y) 
= \int_{-1}^1 W_n^{\n}(xy + \s \sqrt{1 - x^2} \sqrt{1 - y^2}) G_{\n - \half}(d \s).   
\end{equation}

To cope with the drop in index (dimension) in (\ref{eqn:AF-1}), we need the {\it Feldheim-Vilenkin integral} \cite[(2.11)]{bin_rep}, \cite[p.315]{andrewsaskroy}, \cite{askf},
\begin{align}
W_n^{\l}(x) = &\left[ \frac{2 \G(\l + \half)}{\G(\nu + \half) \G(\l - \nu)} \right] \int_0^1 u^{2\nu} (1-u^2)^{\l - \nu - 1} \nonumber \\
&\cdot [x^2 - x^2 u^2 + u^2]^{\half n} W_n^{\n} \left( \frac{x}{\sqrt{x^2 - x^2 u^2 + u^2}} \right) du. \label{eqn:FV}
\end{align}

\section{The result}
\label{result}

We can now formulate our result.

\begin{theorem}\label{thrm:binint}
For $r \in (-1,1)$, the sum of the Askey-Fitch series (\ref{eqn:AF-1}) above is given by the integral (\ref{eqn:astast}) below: 
\begin{equation}\label{eqn:astast}
\int_0^1 H_{\n}^{\l}(du) \int_{-1}^1 G_{\n - \half}(dv) \,
\frac{\left[1 - r^2 (x^2 - x^2 u^2 + u^2) \right]}{I^{\n + 1}},     
\end{equation}
where $I$ is given by
\begin{equation}
I :=
1 - 2r\cdot \frac{xy + uv \sqrt{1 - x^2} \sqrt{1 - y^2}}{\sqrt{x^2 - x^2 u^2 + u^2}}
+ \frac{(xy + uv \sqrt{1 - x^2}\sqrt{1 - y^2})^2}{(x^2 - x^2 u^2 + u^2)}.
\end{equation}
Moreover, this holds also for $r = 1$ unless $\nu < \l \leq \nu + 1$.
\end{theorem}

\begin{proof}
We sum the series by reducing it to the generating function (\ref{eqn:GF}).  There are two steps: reduction of $\l$ to $\nu$ by the Feldheim-Vilenkin integral (\ref{eqn:FV}) and reduction of two $W_n$ terms to one by Gegenbauer's multiplication theorem (\ref{eqn:G}). 

\i We follow \cite{bin_rep}.  As there, we may substitute for $W_n^{\l}(x)$ from (\ref{eqn:FV}) into the series (\ref{eqn:AF-1}) and integrate termwise, rewriting (\ref{eqn:AF-1}) as
\begin{equation}
\int_0^1 H_{\n}^{\l}(du) \sum_{n=0}^{\infty} {\om}_n^{\n} (r [x^2 - x^2 u^2 + u^2]^{\half})^n \cdot
            W_n^{\n}(y) W_n^{\n} \left( \frac{x}{\sqrt{x^2 - x^2 u^2 + u^2}} \right).
\end{equation}
We use Gegenbauer's multiplication theorem (\ref{eqn:G}) with
\begin{equation*}
r \mapsto r\sqrt{x^2 - x^2 u^2 + u^2},
\end{equation*}
and replace the product of $W_n^{\n}$ factors in the above, at the cost of another integration over 
$G_{\n - \half}(dv)$, by a single $W_n^{\n}$ term, with argument
\begin{equation}
\frac{xy}{\sqrt{x^2 - x^2 u^2 + y^2}} + v \sqrt{1 - y^2}.\sqrt{1 - \frac{x^2}{x^2 - x^2 u^2 + u^2}}
= \frac{xy + uv \sqrt{1 - x^2} \sqrt{1 - y^2}}{\sqrt{x^2 - x^2 u^2 + u^2}}.
\end{equation} 
The integrand is now of the form $\sum \om^\nu_n r^n W^\nu_n(\cdot)$, and the result now follows from the generating function (\ref{eqn:GF}).                            
\end{proof}

This result completes and complements the work in \cite{askf} and  \cite{bin_rep} by displaying the dependence on the higher index $\l$ in a structurally revealing way: for simplicity, let $r = 1$ so that
\begin{equation}
I = \left( 1 - \frac{xy + uv \sqrt{1 - x^2}\sqrt{1-y^2}}{\sqrt{x^2 - x^2u^2 + u^2}}\right)^2,
\end{equation}
and (\ref{eqn:astast}) is given by
\begin{equation}
\int_0^1 H_\nu^\l(du) \int_{-1}^1 G_{\nu - \half}(dv) \frac{1 - (x^2 - x^2 u^2 + u^2)}{I^{\nu + 1}}.
\end{equation}
Using the definition of $H^\l_\nu$ and the probability measure $G_{\nu + \half}$ and simplifying, (\ref{eqn:astast}) becomes
\begin{align}
\frac{2}{\sqrt{\pi}} \int_0^1 \frac{\G(\l + \half)}{\G(\l - \nu)} u^{2\nu} (1-u^2)^{\l - \nu - \half} \int_{-1}^1 (1-v^2)^{\nu - \half} \left[ \frac{1 - (x^2 - x^2u^2 + u^2)}{I^{2(\nu + 1)}} \right] \, dv du.
\end{align}
Note that the higher index $\l$ occurs only in the outer integral. Moreover, the interactions between the indexes in the outer integral occurs only in the Gamma function $\G(\l - \nu)$ and the power $\l - \nu - 1/2$ of $(1-u^2)$.

\begin{figure}
% You can include your own figure here using the \includegraphics command,
% or alternatively use the \epsfig, \psfig, or \graphicx packages.
% All lines between here and the \caption statement should then be deleted.
% Use the \vspace command, e.g. \vspace*{5cm}, to leave room for any
% artwork that is provided separately (e.g. as camera-ready copy)
   \vspace*{8pt}%\framebox[8cm][l]{%
   \begin{minipage}{15cm}
   \centering
   \includegraphics[width = 0.8\textwidth]{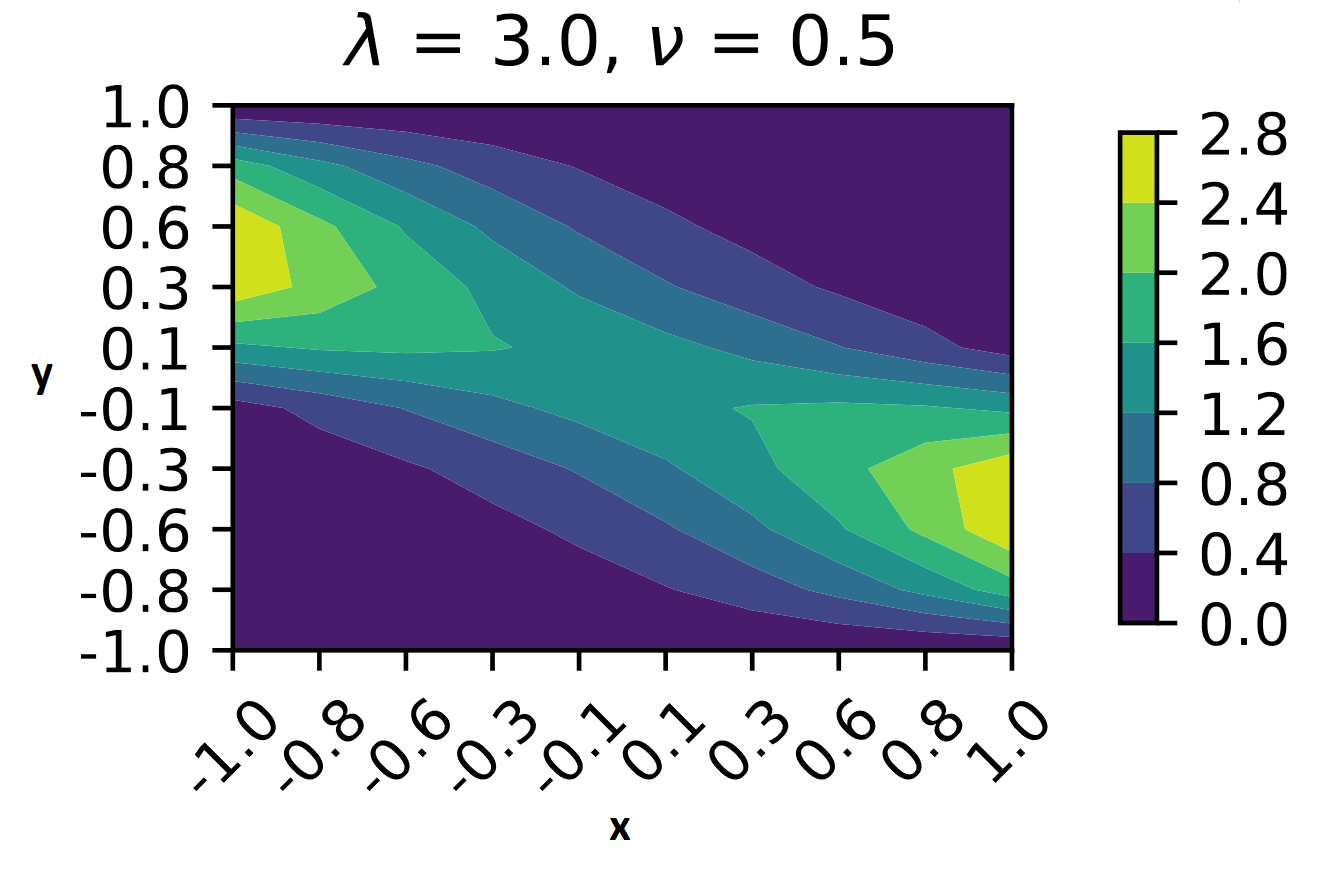}
   \end{minipage}
   \vspace*{8pt}
\caption{Numerical evaluation of series (\ref{eqn:astast}) for $\lambda = 3.0$. $\nu = 0.5$.}
\label{myfigure}
\end{figure}

\section{Dimension walks}
Write ${\cal P}_{\nu}$ for the class of functions $f$ on $[-1,1]$ which are mixtures of $W_n^{\nu}$, i.e., of the form
\begin{equation}\label{eqn:schoenberg}
f(x) = \sum_{n=0}^\infty a_n W_n^{\nu} (x), \quad \sum a_n = 1, a_n \geq 0
\end{equation}
for some probability law $a = \{ a_n \}_0^{\infty}$ (the ultraspherical series converges uniformly as $|W_n^{\nu}(x)| \leq 1$).  The classes ${\cal P}_{\nu}$ are decreasing in $\nu \in [0,\infty]$, and are continuous in $\nu$, in that
\[
\bigcap \{ {\cal P}_{\mu}: 0 \leq \mu < \nu \} = {\cal P}_{\nu}, \qquad
\bigcup \{ {\cal P}_{\mu}: \nu < \mu \leq \infty \} = {\cal P}_{\nu}
\]
(\cite[Th. 1]{bin}).  While the parameters $\l$, $\nu > 0$ are continuous here, the half-integer values are particularly important.  With ${\Sp}^d$ the $d$-sphere -- the unit sphere in Euclidean $(d+1)$-space ${\R}^{d+1}$, a $d$-dimensional Riemannian manifold -- the relevant index for the ultraspherical polynomial is $\nu$, where
\[
\nu = \frac{1}{2}(d - 1).
\]
With $\nu < \l$ as above, the higher dimension corresponding to $\l$ will be written $d'$ (so $\l = \frac{1}{2}(d' - 1)$).  Then, as in \cite{bin_rws}, \cite{bin_rep}, the passage from $\l$ to $\nu < \l$ corresponds to {\it projection} from the $d'$-sphere to the $d$-sphere.  The limiting case $\nu = \infty$ gives $W_n^{\infty} (x) = x^n$, and ${\cal P}_{\infty}$ is the class of probability generating functions, or the class of positive definite functions on the unit sphere in Hilbert space (\cite [Lemma 2]{bin}, \cite{sch}).   

Covariance functions on spheres are very valuable in applications to Planet Earth (see \S 5).  Operations which preserve positive-definiteness are useful in the construction of new families of such covariance functions. Two such operations, coined `walks on dimensions', one changing the dimension by 1, and the other by 2, were proposed for positive-definite functions on spheres by Beatson and zu Castell \cite{beaz2}, \cite{beaz1}. The one-step walks in \cite{beaz2} are based on the Riemann-Liouville operators, but lack the highly desirable {\it semi-group property}, in which passage from $\l$ to $\nu$ and then $\nu$ to $\mu$ is the same as passage from $\l$ to $\mu$ directly.

\begin{theorem}\label{thrm:dimwalks}
For $f \in {\cal P}_{\nu}$ as in (\ref{eqn:schoenberg}),
\begin{equation}
f(x) = \sum_{n=0}^\infty a_n W_n^{\l} (x) \in {\cal P}_{\l}.
\end{equation} 
\end{theorem}

\begin{proof}
\begin{align}
\int_{-1}^1 f(y) M^{\l}_{\nu}(x; dy) &= \int_{-1}^1  \sum_{n=0}^\infty a_n W_n^{\nu} (y)  M^{\l}_{\nu}(x; dy) \\
&= \sum_{n=0}^\infty  a_n \int_{-1}^1  W_n^{\nu} (y)  M^{\l}_{\mu}(x; dy) \label{sumint} \\
&=  \sum_{n=0}^\infty  a_n W^{\l}_n(x) \in \mathcal{P}(\Sp^{d'}),
\end{align}
interchange of the sum and integral in \ref{sumint} being justified by the uniform convergence of the Schoenberg expansion.
\end{proof}

\begin{corollary}  The operation of passing from $f(x) \in {\cal P}_{\nu}$ to $\int_{-1}^1 f(y) M_{\nu}^{\l}(x, dy) \in {\cal P}_{\l}$ in the theorem has the semigroup property.
\end{corollary}

\begin{proof}  The mixture coefficients $a_n$ are unchanged by this operation, and so remain unchanged under further operations of the same type.
\end{proof}

\section{Complements}
\subsection{Hypergroups and symmetric spaces.}  
Hypergroups are `locally compact spaces with a group-like structure on which the bounded measures convolve in a similar way to that on a locally compact group', to quote from the standard work on this important subject, \cite[p.1]{BloH1}.  The probabilistic setting of random walks on spheres \cite[p.196-197]{bin_rws} that inspired both \cite{bin_rep} and this paper, its sequel, is in hypergroup language that of the Bingham (or Bingham-Gegenbauer) hypergroup.  This in turn was inspired by Kingman's work on random walks with spherical symmetry \cite{kin}, which gives the Kingman (or Kingman-Bessel) hypergroup.  The theory for spheres and for spherical symmetry give the prototypical examples of symmetric spaces of rank one of compact type (constant positive curvature) and of Euclidean type (zero curvature); these are complemented by the case of constant negative curvature, the hyperbolic or Zeuner hypergroups \cite{zeuner}.  For background on symmetric spaces we refer to Helgason \cite{Hel2}, for spaces of constant curvature to Wolf \cite{Wol1}, and for compact symmetric spaces to Askey and Bingham \cite{askb}.

We note that the Kingman situation (Euclidean space with spherical symmetry) may be recovered from the spherical one here by letting the radius of the sphere tend to infinity.  The Bessel functions in the Kingman theory arise from radialisation of the Fourier transform in Euclidean space under spherical symmetry \cite[II.7]{Bocc}.

\subsection{Gaussian processes, path properties, Tauberian theorems.}
The positive definite functions in the classes ${\cal P_{\nu}}$ of \S 4 serve as covariances of Gaussian processes parametrised by spheres.  Their distributions are determined by  the sequence $a = \{ a_n \}$ (the {\it angular power spectrum}) of the Schoenberg expansion coefficients above.  In particular, the rate of decay of the $a_n$ governs the path properties: the faster the decay, the smoother the paths.  For details, see  \cite{sheppmem}.  Crucial here is Malyarenko's theorem \cite[Ch. 4]{mal_book}.  This rests on a Tauberian theorem of the first author \cite{bin_jacobi}, which in turn derives from work of Askey and Wainger \cite{askw}.  Here it is necessary to move from the one-parameter family of ultraspherical polynomials $W_n^{\nu}$ to the two-parameter family of Jacobi polynomials $J_n^{\alpha, \beta}$ containing it (\cite[Ch. 6]{andrewsaskroy}, \cite[Ch. IV]{sze}). 

\subsection{Sphere cross line.}
The motivation for much of the interest in positive definite functions on spheres derives from its applications in geostatistics.  Here one has both spatial dependence and temporal evolution, and so one is dealing with geotemporal processes.  For background here, see e.g. \cite{bms}, \cite{bins_walks}.

\section*{Postscript}

To close, the first author takes pleasure in noting the half-century between  Part I \cite{bin_rep} (which derives from his own PhD of 1969) and the present Part II (which derives from the second author's PhD of 2020).  We both take pleasure in dedicating the paper to the memory of Dick Askey, whose influence pervades it.  Dick was a famous expert on special functions, but was interested in their applications, including those to probability.  When \cite{askb} was written, he used to dine out by saying, with tongue in cheek, ``I've just written a paper with Bingham on Gaussian processes -- whatever they are."

% Describe numerical approx.

%\begin{acknowledgements}\label{ackref}
%The \verb"acknowledgements" environment may be used to acknowledge
%indebtedness to colleagues, host institutions and referees. Accounts
%of grants and financial support should be made as a footnote on the
%title page using the \verb"\extraline{}" command in the preamble.
%\end{acknowledgements}

%\begin{thebibliography}{9}% Replace 9 by 99 if 10 or more references
\bibliographystyle{plain}
\bibliography{bibliography}
%\end{thebibliography}

\affiliationone{% in this example, two authors share an institution
   N. H. Bingham\\
   Department of Mathematics \\
   Imperial College London \\
   South Kensington Campus \\
   London, SW7 1AZ \\
   UK 
   \email{n.bingham@imperial.ac.uk}}
% Important: Do not put any empty line here.
\affiliationtwo{% in this example, one author has two addresses}
   Tasmin L. Symons \\
   Telethon Kids Institute \\
   15 Hospital Avenue \\
   Perth, WA 6009 \\
   Australia 
   \email{tasmin.symons@telethonkids.org.au}}
% Important: Do not put any empty line here.
% Use \affiliationthree{} for any address positioned under \affiliationone
% Use \affiliationfour{}  for any address positioned under \affiliationtwo
%\affiliationthree{~} %inserts a space to make this field empty
%\affiliationfour{%
%   Current address:\\
%   Present long-term address\\
%   Country
%   \email{t.hird@institution.edu}}
%
\end{document}